\documentclass[12pt,a4paper]{article}
\usepackage[utf8]{inputenc}
\usepackage[UKenglish]{isodate}
\cleanlookdateon
\usepackage{amssymb}
\usepackage{amsmath}
\usepackage{amsthm}
\usepackage{mathtools}
\usepackage{mdwlist}
\usepackage{enumitem}
\usepackage[usenames,dvipsnames]{xcolor}
\usepackage[T1]{fontenc}
\usepackage{parskip}
\usepackage{esvect}
\usepackage{lipsum}
\usepackage{hyperref}
\usepackage{thmtools}
\usepackage{thm-restate}
\usepackage[sort,compress,numbers]{natbib}

\hypersetup{
  colorlinks   = true,
  urlcolor     = blue!50!black,
  linkcolor    = blue!50!black,
  citecolor   = blue!50!black
}

\newtheorem{theorem}{Theorem}

\newtheorem{lemma}[theorem]{Lemma}
\newtheorem{corollary}[theorem]{Corollary}
\newtheorem{proposition}[theorem]{Proposition}

\newtheorem{conjecture}[theorem]{Conjecture}
\newtheorem{question}[theorem]{Question}
\newtheorem{claim}{Claim}

\author{
Ant\'onio Gir\~ao\footnote{Mathematical Institute, University of Oxford,
Oxford, OX2 6GG, UK}~\footnote{Research supported by EPSRC grant EP/V007327/1} \;
Freddie Illingworth\protect\footnotemark[1]~\footnotemark[2] \;
Emil Powierski\protect\footnotemark[1] \;
\\ Michael Savery\protect\footnotemark[1]~\footnote{Heilbronn Institute for Mathematical Research, Bristol, UK} \;
Alex Scott\protect\footnotemark[1]~\footnotemark[2] \;
Youri Tamitegama\protect\footnotemark[1]\; 
Jane Tan\protect\footnotemark[1]}

\renewcommand{\leq}{\leqslant}
\renewcommand{\geq}{\geqslant}

\newcommand*{\Z}{\mathbb{Z}}

\newcommand*{\cC}{\mathcal{C}}
\newcommand*{\cF}{\mathcal{F}}
\newcommand*{\cG}{\mathcal{G}}
\newcommand*{\cH}{\mathcal{H}}
\newcommand*{\cY}{\mathcal{Y}}
\newcommand*{\ov}[1]{\overline{#1}}
\newcommand*{\defn}[1]{\textcolor{Maroon}{\emph{#1}}}
\DeclarePairedDelimiter{\abs}{\lvert}{\rvert}

\newcommand\blfootnote[1]{
  \begingroup
  \renewcommand\thefootnote{}\footnote{#1}
  \addtocounter{footnote}{-1}
  \endgroup
}

\advance\textwidth2\evensidemargin 
\title{\texorpdfstring{\vspace{-4ex}}{}Induced subgraphs of induced subgraphs of large chromatic number}
\date{18 September 2023}

\begin{document}

\maketitle

\begin{abstract}
\noindent
We prove that, for every graph $F$ with at least one edge, there is a constant $c_F$ such that there are graphs of arbitrarily large chromatic number and the same clique number as $F$ in which every $F$-free induced subgraph has chromatic number at most $c_F$. This generalises recent theorems of Bria\'{n}ski, Davies and Walczak, and Carbonero, Hompe, Moore and Spirkl. 
Our results imply that for every $r\geq 3$ the class of $K_r$-free graphs has a very strong vertex Ramsey-type property, giving a vast generalisation of a result of Folkman from 1970.  We also prove related results for tournaments, hypergraphs and infinite families of graphs, and show an analogous statement for graphs where clique number is replaced by odd girth.

\blfootnote{Email: \textsf{\{\href{mailto:girao@maths.ox.ac.uk}{girao},\href{mailto:illingworth@maths.ox.ac.uk}{illingworth},\href{mailto:powierski@maths.ox.ac.uk}{powierski},\href{mailto:savery@maths.ox.ac.uk}{savery},\href{mailto:scott@maths.ox.ac.uk}{scott},\href{mailto:tamitegama@maths.ox.ac.uk}{tamitegama},\href{mailto:jane.tan@maths.ox.ac.uk}{jane.tan}\}@maths.ox.ac.uk}}
\end{abstract}

\section{Introduction}\label{sec:intro}

Given a graph with large chromatic number, what can we say about its local structure?  For example, are there induced subgraphs that appear in every graph with sufficiently large chromatic number?  Perhaps the first question of this type one might ask is whether every graph with large chromatic number contains a large complete subgraph.  However, it has been known since the 1940s that this is not the case: Tutte~\cite{Desc54} and Zykov~\cite{Z49} gave constructions of triangle-free graphs with arbitrarily large chromatic number, and Erd\H{o}s~\cite{Erdos59highgirth} showed that there exist graphs with both arbitrarily large girth and chromatic number. But this raises another question: if a graph with large chromatic number does not contain a large complete subgraph, then what can we say about its induced subgraphs?

It is helpful to use the language of $\chi$-bounded classes: a hereditary class $\mathcal G$ is \defn{$\chi$-bounded} if there is some function $f$ such that $\chi(G) \leq f(\omega(G))$ for all $G\in \mathcal G$, where $\chi(G)$ denotes the \defn{chromatic number} of $G$ and $\omega(G)$ denotes the \defn{clique number} (the maximum number of vertices in a complete subgraph of $G$). The study of $\chi$-boundedness was strongly influenced by an important paper of Gy\'{a}rf\'{a}s~\cite{Gyarfas87}, that put forward a number of influential conjectures.  There has been a burst of recent activity in this area, and many of the conjectures have recently been resolved (see~\cite{S22,SS20}).

An important strategy for showing that a hereditary class $\cG$ is $\chi$-bounded has been that of `cleaning up' its structure: given $G\in\cG$ with very large chromatic number, we can attempt to find an induced subgraph that still has fairly large chromatic number, but is structurally simpler in some way. If this succeeds for all $G \in \cG$, then to determine whether the class $\cG$ is $\chi$-bounded we need only consider these structurally simpler graphs. A very strong version of this approach was proposed independently by Fred Galvin and Vojt\v{e}ch R\"{o}dl (see~\cite{N79}) and by Louis Esperet (see~\cite{SS20}), who made the striking conjecture that every graph with large chromatic number contains either a large clique or an induced triangle-free subgraph with large chromatic number. If this were true, it would have radical consequences for the study of $\chi$-boundedness, as it would imply that whether or not a class is $\chi$-bounded is determined by the triangle-free graphs in the class. However, Galvin and R\"{o}dl's conjecture was recently disproved by Carbonero, Hompe, Moore and Spirkl~\cite{CHMS22}, who proved the following via a surprising new twist on a construction of Kierstead and Trotter~\cite{KT92}.

\begin{theorem}[Carbonero, Hompe, Moore and Spirkl~\cite{CHMS22}]\label{thm:chms}
There are graphs of arbitrarily large chromatic number that contain neither a $K_4$ nor a triangle-free induced subgraph of chromatic number greater than four.
\end{theorem}

Shortly afterwards, Bria\'{n}ski, Davies and Walczak~\cite{BDW22arxiv} extended the result of 
Carbonero, Hompe, Moore and Spirkl to cliques of prime order in an ingenious paper proving the following.\footnote{Building on this result, they further showed that there are classes of graphs that are $\chi$-bounded but not polynomially $\chi$-bounded, disproving another conjecture of Esperet~\cite{E17}.}

\begin{theorem}[Bria\'{n}ski, Davies and Walczak~\cite{BDW22arxiv}]\label{thm:bdw}
For every prime $p$, there are $K_{p + 1}$-free graphs $G$ of arbitrarily large chromatic number such that every $K_p$-free induced subgraph $H$ of $G$ satisfies $\chi(H)\leq \omega(H)^{\omega(H)^2}$.
\end{theorem}

Theorems~\ref{thm:chms} and \ref{thm:bdw} show that there are $K_{p+1}$-free graphs of large chromatic number in which every induced subgraph of large chromatic number contains a copy of $K_p$.  Thus, the `cleaning up' strategy cannot be deployed in the straightforward way that Esperet's conjecture suggests: we cannot simply drop to an induced subgraph with large chromatic number but smaller clique number.  Nevertheless, we might still hope that we can get rid of some other graphs.

In this paper, we show that Theorems \ref{thm:chms} and \ref{thm:bdw} are just the tip of the iceberg.  Let us say that a graph $G$ is \defn{$F$-free} if it does not contain an induced copy of $F$.  We will show that results such as Theorems \ref{thm:chms} and \ref{thm:bdw} hold not just for cliques, but in fact for every nontrivial graph $F$.

\begin{theorem}\label{thm:main}
For every graph $F$ with at least one edge, there is a constant $c_F$ and graphs $G$ of arbitrarily large chromatic number and the same clique number as $F$ such that every $F$-free induced subgraph of $G$ is $c_F$-colourable.
\end{theorem}

It follows that there is no general `cleaning up' strategy to get rid of induced subgraphs: it is necessary instead to use particular properties of the class under consideration, or to focus on other structural features.\footnote{See, for example, the discussion of diameter and Conjecture 1.10 in~\cite{SS20Banana}.}

Theorem~\ref{thm:main} is also interesting from the perspective of Ramsey theory.
A class $\cC$ of graphs is a \defn{vertex Ramsey class} if, for every $F \in \cC$ and positive integer $k$, there is some $G \in \cC$ such that every vertex colouring of $G$ in $k$ colours contains a monochromatic induced copy of $F$.
Vertex Ramsey classes were first studied in the 1970s by Folkman~\cite{F70} and Ne\v{s}et\v{r}il and R\"odl~\cite{NR76} as part of the development of a rich theory (see~\cite{Nesetril96survey}).
The general question is to determine which hereditary classes $\cC$ of graphs are vertex Ramsey. There has been substantial progress, particularly in the case when $\cC$ is determined by a single excluded induced subgraph, that is, when $\cC$ is the class of $F$-free graphs for some $F$ (see~\cite{F70,NR76,RSZ95,RS92,S95,SZ92}), or when $\cC$ is defined by excluding a clique and a tree~\cite{Kierstead97,KZ96}. Ramsey classes of infinite graphs have also been studied in~\cite{HK88,KS93,KR86}.

One of the earliest results on vertex Ramsey classes is due to Folkman~\cite{F70}, who proved in 1970 that the class $\cC_r$ of $K_{r}$-free graphs is vertex Ramsey for every $r \geq 3$: in other words, for every $K_r$-free graph $F$ and every positive integer $k$, there is a $K_r$-free graph $H$ such that in every $k$-colouring of $V(H)$ some colour class contains an induced copy of $F$. Theorem~\ref{thm:main} proves a vastly stronger property: it shows that for every $K_{r}$-free  graph $F$ there are $K_{r}$-free graphs with arbitrarily large chromatic number in which \emph{every} induced subgraph of large chromatic number contains an induced copy of $F$. This difference resembles the difference between the theorems of Van der Waerden~\cite{VdW27} and Szemer\'{e}di~\cite{Szem75} on monochromatic arithmetic progressions in the integers.

Interestingly, there is no edge version of Theorem~\ref{thm:main}: R\"{o}dl~\cite{rodl77} proved that, for all $k$, every graph with sufficiently large chromatic number contains a triangle-free subgraph with chromatic number $k$. In particular, while the class of triangle-free graphs is both vertex and edge Ramsey~\cite{NR76edge} and satisfies the strengthened vertex Ramsey property, it does not satisfy the analogous strengthened edge Ramsey property.

If $F$ is triangle-free, then Theorem~\ref{thm:main} guarantees a triangle-free $G$. It is natural to ask whether one can go further and guarantee that $G$ has the same girth as $F$. We conjecture that this is the case, but have been unable to prove it. However, we can prove the following weaker statement, where girth is replaced by \defn{odd girth} (the length of the shortest odd cycle). This implies that the class $\cC'_r$ of $\{C_3, C_5, \dotsc, C_{2r + 1}\}$-free graphs has the strengthened vertex Ramsey property mentioned above.

\begin{theorem}\label{thm:oddgirth}
    For every nonbipartite graph $F$, there is a constant $c'_F$ and graphs $G$ of arbitrarily large chromatic number and the same odd girth as $F$ such that every $F$-free induced subgraph of $G$ is $c'_F$-colourable.
\end{theorem}

Theorems~\ref{thm:main} and \ref{thm:oddgirth} both easily extend to finite families of graphs. Indeed, if $\cF$ is a finite family of graphs (at least one of which contains an edge) and $G_{\cF}$ is the disjoint union of the members of $\cF$, then applying Theorem~\ref{thm:main} to $G_{\cF}$ shows that there is some constant $c_{\cF}$ and graphs $G$ of arbitrarily large chromatic number and the same clique number as the largest clique number of a graph in $\cF$ such that every induced subgraph of $G$ with chromatic number greater than $c_{\cF}$ contains every member of $\cF$ as an induced subgraph. Does this phenomenon occur for infinite families? We will say that a family of graphs $\cF$ is \defn{durable} if there exists an infinite graph $H$ with infinite chromatic number such that every infinite-chromatic induced subgraph of $H$ contains every member of $\cF$ as an induced subgraph. All finite families of graphs are durable but the situation for infinite families is more complicated. We obtain an effective test for deciding whether a family of unbounded chromatic number is durable which, for example, gives the following result.
\begin{theorem}\label{thm:girthbad}
    For every integer $g$, the family of graphs with girth at least $g$ is not durable.
\end{theorem}
In the bounded chromatic number case we obtain a link to $\chi$-boundedness, which in particular implies the following.
\begin{theorem}\label{thm:pathdurable}
    The family of paths is durable.
\end{theorem}

We also extend and generalise Theorem~\ref{thm:main} to the settings of tournaments and hypergraphs (see Sections~\ref{sec:tournaments} and \ref{sec:hypergraphs} for definitions not given below). For tournaments we prove the following.

\begin{restatable}{theorem}{trnmt}\label{thm:trnmt}
    For every tournament $T$, there is a constant $C_T$ and tournaments $S$ of arbitrarily large chromatic number such that every $T$-free subtournament of $S$ has chromatic number at most $C_T$.
\end{restatable}

For hypergraphs, there
are two notions for chromatic number: the usual one where no edge is monochromatic and the strong one where no two vertices in an edge have the same colour. Pleasingly, our generalisation of Theorem~\ref{thm:main} uses the usual notion of chromatic number for the constructed hypergraph, but the strong notion for its induced subhypergraphs. The statement also has a very general analogue of a clique. A hypergraph is said to be \defn{strongly $t$-colourable} if its vertices can be $t$-coloured such that no edge contains two vertices of the same colour. A hypergraph \defn{covers} a pair of vertices if it has an edge containing both of them. Theorem~\ref{thm:hypergraphs} applies to general (not necessarily uniform) hypergraphs, but we assume throughout this paper that all edges have size at least two. 

\begin{restatable}{theorem}{hypergraphs}\label{thm:hypergraphs}
    For every hypergraph $\cF$ with at least one edge, there is a constant $c_{\cF}$ and hypergraphs $\cG$ of arbitrarily large chromatic number such that every $\cF$-free induced subhypergraph of $\cG$ is strongly $c_{\cF}$-colourable. Moreover, we can take $\cG$ such that if $\cG$ covers every pair from some $X \subseteq V(\cG)$, then $\cG[X]$ is an induced subhypergraph of~$\cF$.
\end{restatable}

It follows from the moreover part of the theorem that the edge sizes which appear in $\cG$ are the same as those which appear in $\cF$, and that for all $k$ the maximum size of a $k$-uniform clique in $\cG$ is the same as that in $\cF$. In particular, this theorem contains Theorem~\ref{thm:main}.

Our constructions for Theorems~\ref{thm:main}, \ref{thm:oddgirth} and~\ref{thm:hypergraphs} all use a common framework: we start with a base structure of large chromatic number and add edges according to `distances' in the base structure. Using generalised Sidon sets, we are able to choose these `distances' to encode $F$ or $\cF$ and give precise control over where its copies appear. 

The rest of this paper is structured as follows. In Section~\ref{sec:base} we describe the base graphs used in the proofs of Theorems~\ref{thm:main} and~\ref{thm:oddgirth}, then we complete the proofs of these theorems in Sections~\ref{sec:mainproof} and~\ref{sec:oddgirth} respectively. Section~\ref{sec:tournaments} is concerned with tournaments and contains the proof of Theorem~\ref{thm:trnmt}, then in Section~\ref{sec:hypergraphs} we turn our attention to hypergraphs and prove Theorem~\ref{thm:hypergraphs}. Our results on the durability of infinite families of graphs are covered in Section~\ref{sec:infinite}, before we close with a discussion of some open problems in Section~\ref{sec:conclusion}.

\textbf{Note added.} Since the original version of this paper was made available, Bria\'nski, Davies and Walczak have independently proved a strengthening of their Theorem~\ref{thm:bdw}, improving the bound on $\chi(H)$ and removing the restriction that $p$ is prime.  This can be found in the updated version of their paper~\cite{BDW23}.

\section{Orientations of classical constructions}\label{sec:base}

Our constructions for Theorems~\ref{thm:main} and~\ref{thm:oddgirth} start with orientations of graphs with high chromatic number and restricted clique number or girth. Traditionally the underlying graphs in such orientations have been the well known Zykov graphs~\cite{Z49}. Orientations of these were first considered by Kierstead and Trotter~\cite{KT92}, and then used in the constructions of 
both Carbonero, Hompe, Moore and Spirkl~\cite{CHMS22} and
Bria\'{n}ski, Davies and Walczak~\cite{BDW22arxiv}. For the proof of Theorem~\ref{thm:main}, we require precisely the following properties of the underlying graphs $X_n$ and their orientations $\vv{X_n}$:
\begin{enumerate}[label=(\alph*)]
    \item $\chi(X_n) = n$,
    \item $\vv{X_n}$ is acyclic, that is, it contains no directed cycle, and
    \item for all $u, v\in V(\vv{X_n})$, there is at most one directed path from $u$ to $v$ in $\vv{X_n}$.
\end{enumerate}
The Zykov graphs and their orientations given by Kierstead and Trotter satisfy these conditions and so suffice for the proof of Theorem~\ref{thm:main}. 
In fact   any $X_n$ with chromatic number $n$ and girth greater than $2n - 2$ has some orientation $\vv{X_n}$ satisfying these three properties (see Lemma~\ref{lem:digraphexistence}).

For the proof of Theorem~\ref{thm:oddgirth} we require one further property in order to rule out short odd cycles (in particular, we cannot take Zykov graphs as the base graphs in the proof of Theorem~\ref{thm:oddgirth}).  Define a \defn{change of direction} in a digraph to be a pair of edges incident to a common vertex that are both directed towards or both directed away from that vertex. For a given positive integer $g$ we ask that $\vv{X_n}$ satisfies:

\begin{enumerate}[label=(\alph*), resume]
    \item for every cycle in $X_n$, the corresponding oriented cycle in $\vv{X_n}$ has at least $g$ changes of direction.
\end{enumerate}
Strictly speaking, $\vv{X_n}$ and condition (d) depend on the choice of positive integer $g$, but we suppress this dependence in the notation. Note that if $g\geq 3$, then property (d) implies both (b) and (c). Lemma~\ref{lem:digraphexistence} below states that any $X_n$ with chromatic number $n$ and girth greater than $(g - 1)(n - 1)$ where $g \geq 3$ has an orientation $\vv{X_n}$ satisfying all four properties.

By property (c), it is natural to view $V(\vv{X_n}) = V(X_n)$ as a poset with a strict partial order based on \defn{reachability}: for distinct vertices $u$ and $v$, we write $u < v$ if there is a directed path from $u$ to $v$ in $\vv{X_n}$. Note that there is a correspondence between chains in this poset and sets of vertices that lie on a single directed path in $\vv{X_n}$. For vertices $u < v$ we can define the \defn{distance} between $u$ and $v$, written $d(u,v)$,
to be the length of the unique directed path from $u$ to~$v$. The uniqueness of directed paths ensures that $d(u,w)=d(u,v)+d(v,w)$ for vertices $u < v < w$.

At this stage, the reader may like to proceed directly to Sections~\ref{sec:mainproof} and \ref{sec:oddgirth} for the proofs of our main theorems, keeping in mind that there exist (di)graphs with the required properties. We continue in this section with the promised constructions. 

The following proposition, which is the acyclic case of the Gallai-Roy-Hasse-Vivater theorem, provides a strong relationship between orientations and chromatic number.

\begin{proposition}\label{prop:boundedpaths}
    A graph $L$ is $k$-colourable if and only if it has an acyclic orientation $\vv{L}$ which contains no directed paths of length $k$. 
\end{proposition}

\begin{proof}
    First suppose that $L$ is $k$-colourable and let $c \colon V(L) \to \{1, \dotsc, k\}$ be a proper $k$-colouring. For each edge $e = uv$ of $L$, if $c(u) < c(v)$, then orient $e$ from $u$ to $v$ and if $c(v) < c(u)$, then orient $e$ from $v$ to $u$. For any directed path $u_1 u_2 \dotsc u_t$ in the resulting orientation $\vv{L}$, we have $1 \leq c(u_1) < c(u_2) < \dotsb < c(u_t) \leq k$ and so the orientation is acyclic and the path has length at most $k - 1$.
    
    On the other hand, if the acyclic orientation $\vv{L}$ has no directed paths of length $k$, then we may simply colour each vertex $v$ of $L$ by the length of the longest directed path in $\vv{L}$ ending at $v$.
\end{proof}

Erd\H{o}s~\cite{Erdos59highgirth} showed that for any integers $g, n \geq 2$ there are graphs of chromatic number $n$ and girth at least $g$. This together with the following lemma (pointed out to us by Jarik Ne\v{s}et\v{r}il) shows that there are indeed digraphs with properties (a) to (d).

\begin{lemma}\label{lem:digraphexistence}
    Let $g \geqslant 3$ and $n \geqslant 2$ be integers. If a graph $X_n$ has chromatic number $n$ and girth greater than $(g - 1)(n - 1)$, then $X_n$ has an orientation $\vv{X_n}$ satisfying properties \textnormal{(a)}, \textnormal{(b)}, \textnormal{(c)} and \textnormal{(d)}.
\end{lemma}

\begin{proof}
    The graph $X_n$ has chromatic number $n$ so, by Proposition~\ref{prop:boundedpaths}, has an acyclic orientation $\vv{X_n}$ which contains no directed paths of length $n$. In particular, $\vv{X_n}$ automatically satisfies properties (a) and (b). Consider a cycle $C$ in $X_n$. The length of $C$ is greater than $(g - 1)(n - 1)$ and every directed path in $\vv{X_n}$ has length at most $n - 1$, so $C$ has at least $g$ changes of direction in $\vv{X_n}$. Hence property (d) is satisfied. Finally, as $g \geq 3$, this implies that $\vv{X_n}$ satisfies property (c) also.
\end{proof}

\section{Clique number}\label{sec:mainproof}

In this section, we will give the proof of Theorem \ref{thm:main}.  
We have not tried hard to optimise the bounds; however, writing $f = \abs{V(F)}$, the proof gives\footnote{Using results from the updated version of Bria\'nski, Davies and Walczak's paper~\cite{BDW23} (see the note at the end of the introduction), this can be improved to $c_F=O(f^9)$ by noting that the graph $G$ we construct in Section~\ref{sec:construction} can be taken to be a subgraph of their $G_{n,p}$.}
\begin{equation}\label{eq:bounds}
    \begin{aligned}
    c_F & \leq (32f^4)^{\abs{E(F)}}, \\
    c_F & \leq \big((2 + o(1))f^4\big)^{\abs{E(F)}}.
    \end{aligned}
\end{equation}
We begin by sketching the proof strategy, and then discuss $B_3$-sets in Section~\ref{sec:B3}, before giving the details of the proof in Section~\ref{sec:construction}.

The constructions in~\cite{BDW22arxiv,CHMS22, KT92} start with oriented (Zykov) graphs $(\vv{X_n})$ satisfying properties (a) through (d), and add new edges $\vv{uv}$ (or sometimes $\vv{vu}$ in~\cite{CHMS22}) between vertices $u < v$ whenever $p$ does not divide $d(u,v)$ for some prime $p$. These papers use arbitrary primes $p$, $p=3$ and $p=2$ respectively. Our construction is similar in that, for some prime $p$, we add edges to $\vv{X_n}$ based on the residues modulo $p$ of the distances between endpoints. However, the set of residues for which we add edges is now more sophisticated.
Let $\vv{G}$ be the digraph so obtained, and let $G$ be its underlying undirected graph. As $X_n$ is a subgraph of $G$, we know that $\chi(G) \geq n$. The extra edges will be added so that the following properties hold:
\begin{itemize}[noitemsep]
    \item $\omega(G) = \omega(F)$, 
    \item $\vv{G}$ is acyclic, and
    \item if there is a long (depending only on $F$) directed path in $\vv{G}$ all of whose edges correspond to the same distance modulo $p$, then some vertices of this path induce a copy of $F$ in $G$.
\end{itemize}
The second and third properties are particularly useful in light of Proposition~\ref{prop:boundedpaths}.

Now taking $G$ as above, suppose that $H$ is an induced subgraph of $G$ which does not contain $F$ as an induced subgraph. Let $H_i$ be the subgraph of $H$ which consists of those edges corresponding to distance $i$ modulo $p$. By the third bullet point, $H_i$ cannot contain a long directed path. Thus, Proposition~\ref{prop:boundedpaths} and the second bullet point imply that $H_i$ has chromatic number bounded in terms of $F$. Taking a product colouring (over the possible $i$) will then show that $H$ itself must have chromatic number at most some $c_F$.

\subsection{\texorpdfstring{$B_3$}{B3}-sets}\label{sec:B3}
To guarantee that $G$ has the same clique number as $F$ our construction will utilise $B_3$-sets. A set of integers $S=\{a_1 < a_2 < \dotsb < a_k\}$ is a \defn{$B_3$-set} if the sums
\begin{equation*}
    a_{i_1}+a_{i_2}+ a_{i_3}, \qquad 1\leq i_1 \leq i_2 \leq i_3 \leq k
\end{equation*}
are all different. A simple example of a $B_3$-set is the set of powers of four --- it would suffice to use an initial segment of this for our arguments, although the resulting bound on $c_F$ would be much worse.
$B_3$-sets (and more generally $B_h$-sets, which we discuss in Section~\ref{sec:Bh}) were introduced by Bose and Chowla~\cite{BC62} as a generalisation of Sidon sets~\cite{S32}, which are precisely the $B_2$-sets (i.e.\ sets of integers where all pairwise differences are distinct). 
In particular, every $B_3$-set is also a $B_2$-set. Let $S$ be a $B_3$-set and let
\begin{equation*}
    D = \{s - s' \colon s' < s \text{ and } s, s' \in S\}
\end{equation*}
be the \defn{difference set of $S$}.
The sets $S$ and $D$ satisfy the following two claims which motivate our interest in $B_3$-sets.
\begin{claim}[Triangle fact]\label{claim:triangle}
    Suppose that $D$ contains not necessarily distinct $d_1, d_2$ such that $d_1 + d_2 \in D$. Then there are $b_1 < b_2 < b_3$ all in $S$ such that $\{b_2-b_1,b_3-b_2\}=\{d_1,d_2\}$.
\end{claim}

\begin{proof}
    Suppose that $d_1, d_2, d_1 + d_2$ are all in $D$. Then there are $x_1, x_2, x_3, y_1, y_2, y_3$ all in $S$ with $d_1 = y_1 - x_1$, $d_2 = y_2 - x_2$ and $d_1 + d_2 = y_3 - x_3$. Hence,
    \begin{align*}
        (y_1 - x_1) + (y_2 - x_2) & = y_3 - x_3, \\
        \Rightarrow y_1 + y_2 + x_3 & = x_1 + x_2 + y_3.
    \end{align*}
    But $S$ is a $B_3$-set, so $y_1, y_2, x_3$ must be $x_1, x_2, y_3$ in some order. The elements of $D$ are positive integers so $x_i \neq y_i$ for all $i$. Hence $y_1$ is either $x_2$ or $y_3$.
    \begin{itemize}[noitemsep]
        \item If $y_1 = x_2$, then $b_1 = x_1$, $b_2 = y_1 = x_2$, and $b_3 = y_2$ satisfy $(b_2 - b_1, b_3 - b_2)  = (d_1, d_2)$.
        \item If $y_1 = y_3$, then $y_2 = x_1$ (as $y_2 \neq x_2$) and so $b_1 = x_2$, $b_2 = x_1 = y_2$, and $b_3 = y_1$ satisfy $(b_2 - b_1, b_3 - b_2)  = (d_2, d_1)$.\qedhere
    \end{itemize}
\end{proof}

The preceding claim can be leveraged to cover more distances.

\begin{claim}[Clique fact]\label{claim:clique}
Let $\ell$ be a natural number and suppose that $D$ contains not necessarily distinct $d_1,\dots,d_\ell$ such that every sum of the form 
\begin{equation}\label{eq:sums}
    \sum_{i_1\leq j \leq i_2}d_j
\end{equation}
with $1\leq i_1 \leq i_2\leq \ell$ is in $D$. Then there are $b_1<\dotsb <b_{\ell+1}$ all in $S$ such that either $b_{i + 1} - b_i = d_i$ for all $i$, or $b_{i + 1} - b_i = d_{\ell+1-i}$ for all $i$.
\end{claim}

\begin{proof}
We proceed by induction on $\ell$. The case $\ell = 1$ is trivial and the case $\ell = 2$ is the triangle fact, so suppose $\ell \geq 3$. By the induction hypothesis, there are $b_1<\dots<b_\ell$ such that either $b_{i + 1} - b_i = d_i$ for all $i<\ell$, or $b_{i + 1} - b_i = d_{\ell-i}$ for all $i<
\ell$. 
We may assume that $b_{i + 1} - b_i = d_i$ for all $i<\ell$, as the other case follows by a symmetric argument.

Let us say that a difference $d \in D$ is associated with $b\in S$ if there is some $c\in S$ such that $\abs{b - c}=d$. As $S$ is a Sidon set, $d_{\ell - 1}$ is only associated with $b_{\ell - 1}$ and  $b_\ell$. Thus, applying the triangle fact to $d_{\ell - 1}$ and $d_\ell$ we find that there is $b \in S$ such that either $b_{\ell - 1} - b = d_\ell$ or $b - b_\ell = d_\ell$. In the latter case we are done, so assume that $b_{\ell - 1} - b = d_\ell$.

Similarly, $d \coloneqq \sum_{j = 1}^{\ell - 1} d_j$ is only associated with $b_1$ and $b_\ell$, so by applying the triangle fact to $d$ and $d_\ell$ we find that there is some $b' \in S$ such that either $b_1 - b' = d_{\ell}$ or $b' - b_{\ell} = d_{\ell}$. In the latter case we are done so we assume that $b_1 - b' = d_{\ell}$. But then $b' < b_1 < b_{\ell - 1}$ are all associated with $d_{\ell}$ while $S$ is a Sidon set. This is a contradiction, as required.
\end{proof}

The existence of large $B_3$-sets is guaranteed by the following theorem of Bose and Chowla~\cite[Thm.~2]{BC62}.
We use the standard shorthand $[n] = \{1,\dotsc,n\}$.

\begin{theorem}[Bose and Chowla~\cite{BC62}]\label{thm:BoseChowla}
    Let $q$ be one more than a prime power. There is a $B_3$-subset of $[q^3]$ of size $q$ that contains both $1$ and $2$.
\end{theorem}
Given a natural number $f$, using the well known fact that the gap between consecutive primes $q < q'$ is $o(q)$ (see, for instance,~\cite{BHP01} for an up-to-date bound), Theorem~\ref{thm:BoseChowla} implies that for $N = (1 + o(1))f^3$, there is a $B_3$-subset of $[N]$ of size at least $f$ that contains both $1$ and $2$. Bertrand's postulate gives a non-asymptotic bound: the smallest prime $p = q - 1$ greater than $f - 1$ is less than $2f - 1$ and so Theorem~\ref{thm:BoseChowla} gives a $B_3$-subset of $[q^3] \subseteq [8f^3]$ of size at least $f$ that contains both $1$ and $2$.

\subsection{Proof of Theorem~\ref{thm:main}}\label{sec:construction}

Fix a graph $F$ with at least one edge and let $f = \lvert V(F) \rvert$. Let $p$ be a prime number chosen to be large enough that $\big[(p - 1)/2\big]$ has a $B_3$-subset $S$ of size $f$ that contains $1$ and $2$. By the discussion at the end of Section~\ref{sec:B3}, it suffices to have $(p - 1)/2 \geq \min\bigl\{(1+o(1))f^3, 8f^3\bigr\}$. Using the result on prime gaps and Bertrand's postulate we may choose such a $p$ with $p\leq \min\bigl\{(2 + o(1))f^3,32f^3\bigr\}$.

Let the elements of $S$ be $s_1 = 1, s_2 = 2, s_3,\dotsc, s_f$ in ascending order and let the difference set of $S$ be $D = \{s_j - s_i \colon i < j\}$. As $F$ has at least one edge we may take a copy $F^{\ast}$ of $F$ with vertex set $S$ which includes the edge $s_1s_2$. We write
\begin{equation*}
    E = \{s_j - s_i \colon i < j \text{ and } s_i s_j \in E(F^{\ast})\}
\end{equation*}
for the set of distances corresponding to edges of $F^{\ast}$. Since $s_1 = 1$ and $s_2 = 2$, we have $1 \in E$. Note that $E \subseteq D$. Let $\ov{D} = D + p\Z$ and $\ov{E} = E + p\Z$.

For a fixed large integer $n$, recall the (di)graphs $\vv{X_n}$ and $X_n$ from Section~\ref{sec:base}, and that for vertices $u < v$ (in the reachability ordering) of these graphs we write $d(u, v)$ for the length of the unique directed path from $u$ to $v$ in $\vv{X_n}$. We now define a digraph $\vv{G}$ with vertex set $V(\vv{X_n})$ as follows.
\begin{center}
    The edge $\vv{uv}$ is present if $u < v$ and $d(u, v) \in \ov{E}$.
\end{center}
The digraph $\vv{G}$ is acyclic, since the presence of the edge $\vv{uv}$ implies that $u < v$ in the reachability order. Let $G$ be the underlying undirected graph of $\vv{G}$. 

We now verify that $G$ satisfies the conditions of Theorem~\ref{thm:main}. Since $1 \in E$, the graph $X_n$ is a subgraph of $G$ and so $\chi(G) \geq n$. Next we show that $\omega(G)\leq \omega(F)$. To this end, suppose that vertices $v_1,\dots,v_{\ell+1}$ form a clique in $G$. Since any two of these vertices are $<$-comparable, we may assume that $v_1<\dotsb<v_{\ell+1}$. For $i\in [\ell]$ let $d_i\in\{0,\dots,p-1\}$ be the residue modulo $p$ of $d(v_i,v_{i+1})$, and note that
\begin{equation*}
    d(v_{i_1}, v_{i_2+1}) \equiv \sum_{i_1\leq j\leq i_2} d_j \bmod{p}
\end{equation*}
for all $1\leq i_1\leq i_2\leq \ell$. Since $v_1,\dots,v_{\ell+1}$ form a clique in $G$, all these distances $d(v_{i_1}, v_{i_2+1})$, and hence all sums of the form in~\eqref{eq:sums}, are in $\ov{E}\subseteq\ov{D}$. We claim that they are actually in $D$. It is clear that $d_j\in D$ for all $j$, and since $D\subseteq [(p-1)/2]$ we have that the two-term sums of the form in~\eqref{eq:sums} are in $D$ as well. Iterating this shows that all sums of the form in~\eqref{eq:sums} are in $D$. 

It now follows from the clique fact that there are $b_1 < \dotsb < b_{\ell+1}$ all in $S$ such that either $b_{i + 1} - b_i = d_i$ for all $i$, or $b_{i + 1} - b_i = d_{\ell + 1 - i}$ for all $i$. All pairwise distances between elements of $B = \{b_1, \dotsc, b_{\ell+1}\}$ are of the form in \eqref{eq:sums} and so are in $\ov{E}$. But since $B \subseteq S$, they must all actually be in $E$. By the definition of $E$ and the fact that $S$ is a Sidon set, the vertices $b_1, \dotsc, b_{\ell+1}$ induce a clique in $F^{\ast}$ and so $\ell+1 \leq \omega(F)$, as required.

We now turn to the last required property of $G$, that all of its induced subgraphs of chromatic number greater than some $c_F$ contain an induced copy of $F$. Start by colouring each edge $uv$ of $G$ (and $\vv{G}$), where $u < v$, with the colour $d(u, v) \bmod{p}$. Let $\vv{P}$ be a directed path in $\vv{G}$ of length $pf$ which is monochromatic in colour $i$ for some $i \in E$. Label the first vertex of this path by $v_1$. Recall that the elements of $S$ are $s_1 < \dotsb < s_f$. Let $v_2$ be the first vertex on path $\vv{P}$ such that $d(v_1,v_2)\equiv s_2-s_1\bmod{p}$. Since $p$ is prime and $i\neq 0$, there are at most $p-1$ edges of $\vv{P}$ between $v_1$ and $v_2$ on this path. Next, define $v_3$ to be the first vertex of $\vv{P}$ such that $d(v_2,v_3) \equiv s_3-s_2\bmod{p}$, and continue in this manner to define $v_4,\dotsc,v_f$. Note that for each $i$ there are at most $p-1$ edges between $v_i$ and $v_{i + 1}$ on path $\vv{P}$, so the path is long enough that these vertices can be found. 

The way in which we have chosen the vertices $v_i$ means that, for each pair $i < j$, we have $d(v_i, v_j) \equiv s_j - s_i \bmod{p}$. Then, by the construction of $G$ and the fact that $S$ is a Sidon set, the edge $v_i v_j$ is present in the graph exactly when there is an edge between $s_i$ and $s_j$ in $F^{\ast}$. It follows that $G$ contains an induced copy of $F$ with vertex set $\{v_1,\dotsc,v_f\}$. 

Now let $H$ be an induced subgraph of $G$ which does not contain an induced copy of $F$, and let $\vv{H}$ be the subdigraph of $\vv{G}$ induced by the same set of vertices. For each $i \in E$, let $H_i$ be the subgraph of $H$ consisting of the $i$-coloured edges and define $\vv{H_i}$ similarly. By the above, $\vv{H_i}$ contains no directed paths of length $pf$ and hence $H_i$ is $pf$-colourable by Proposition~\ref{prop:boundedpaths}. Take such a proper colouring $\chi_i \colon V(H_i) \to [pf]$ for each $i$. Then the product colouring $\chi\colon V(H) \to [pf]^{\abs{E}}$ given by $\chi(v) = (\chi_i(v) \colon i\in E)$ is a proper colouring of $H$. Hence, we may take $c_F$ to be $(pf)^{\abs{E}}$. Finally, we note that $\abs{E} = \abs{E(F)}$ and recall that $p\leq \min\bigl\{(2 + o(1))f^3, 32f^3\bigr\}$ to obtain the bounds given by~\eqref{eq:bounds}.

\section{Odd girth}\label{sec:oddgirth}

The proof of Theorem~\ref{thm:oddgirth} shares some ingredients with that of Theorem~\ref{thm:main}. In this exposition, we will pay particular attention to the differences.

\subsection{\texorpdfstring{$B_h$}{Bh}-sets}
\label{sec:Bh}
Given a natural number $h \geq 2$, we can define the notion of $B_h$-sets analogously to $B_3$-sets. That is, a set of integers $S = \{a_1 < a_2 < \dotsb < a_k\}$ is a \defn{$B_h$-set} if the sums
\begin{equation*}
a_{i_1} + a_{i_2} + \dotsb + a_{i_h}, \qquad 1 \leq i_1 \leq i_2 \leq \dotsb \leq i_h \leq k
\end{equation*}
are all different. Note that a $B_h$-set is also a $B_{h'}$-set for any $h'<h$, and in particular is a Sidon set. The powers of $h$ form a $B_h$-set but Bose and Chowla~\cite[Thm.~2]{BC62} provide a more efficient construction.
\begin{theorem}[Bose and Chowla~\cite{BC62}]
    Let $h$ be a positive integer and $q$ be one more than a prime power. There is a $B_h$-subset of $[q^h]$ of size $q$ that contains both $1$ and $2$.
\end{theorem}

The importance of $B_h$-sets for our proof of Theorem~\ref{thm:oddgirth} is the following claim. We will only make use of the `moreover' part but we need the main statement to apply induction. Let $S$ be a $B_h$-set and let $D$ be its difference set as before. By a \defn{circuit}, we mean a closed walk in which vertices may be repeated but edges may not. 

\begin{claim}[Cycle fact]\label{claim:cycle}
Let $2\leq \ell\leq h$ be integers. Suppose that $D$ contains not necessarily distinct $d_1,\dotsc,d_\ell$ such that for some $1\leq s\leq \ell - 1$ we have
\begin{equation}\label{eq:cycle_fac}
    d_1 + \dotsb + d_s = d_{s+1} + \dotsb + d_\ell.
\end{equation}
Let $M$ be the multigraph on vertex set $S$ built by adding, for each $1\leq i\leq \ell$, an edge $e_i$ between the pair of vertices whose difference is $d_i$. Then the edge set of $M$ can be decomposed into a collection of circuits. Moreover, if $\ell$ is odd, then the graph obtained from $M$ by deleting duplicate edges contains an odd cycle.
\end{claim}

\begin{proof}
We prove that the edge set of $M$ can be decomposed into a collection of circuits by induction on $\ell$. If $\ell=2$, then we have $d_1=d_2$ and $M$ is just a pair of edges between the same two vertices. 

Suppose that $\ell \geq 3$. Let $x_1, \dotsc, x_\ell, y_1, \dotsc, y_\ell \in S$ be such that $d_i = y_i - x_i$ for all $1 \leq i \leq \ell$. From equation~\eqref{eq:cycle_fac} we obtain 
\begin{equation*}
    y_1 + \dotsb + y_s + x_{s+1} + \dotsb + x_\ell = x_1 + \dotsb + x_s + y_{s+1} + \dotsb + y_\ell
\end{equation*}
and so, as $S$ is a $B_h$-set and $\ell \leq h$,  $y_1, \dotsc, y_s, x_{s+1}, \dotsc, x_\ell$ must be $x_1, \dotsc, x_s, y_{s+1}, \dotsc, y_\ell$ in some order. In particular, $y_1 = x_i$ for some $1\leq i \leq s$ or $y_1=y_j$ for some $s+1\leq j \leq \ell$.  In the first case we may assume by relabelling that $y_1 = x_2$ ($y_1 \neq x_1$ as $0 \not\in D$). Then $d \coloneqq d_1 + d_2 = y_2 - x_1$ which is in $D$ as $y_2 > x_2 = y_1 > x_1$. Moreover, $d + d_3 + \dotsb +d_s = d_{s+1} + \dotsb + d_\ell$. Thus, by induction, the edges of the multigraph obtained from $M$ by deleting $e_1$ and $e_2$ and adding an edge $e$ between $x_1$ and $y_2$ can be decomposed into a collection of circuits. Adding back the edges $e_1$ and $e_2$ (which form a path between  $x_1$ and $y_2$) in place of $e$, we obtain a suitable decomposition of the edges of $M$.

In the second case we may assume by relabelling that $y_1 = y_{s + 1}$. If $x_1 = x_{s + 1}$, then $d_1 = d_{s + 1}$ and hence $d_2 + \dotsb + d_s = d_{s + 2} + \dotsb + d_\ell$. By induction, the edges of $M$ without $e_1$ and $e_{s + 1}$ can be decomposed into a collection of circuits. Adding the circuit consisting of $e_1$ and $e_{s + 1}$, we obtain a decomposition of all edges of $M$. 
If $x_1 \neq x_{s + 1}$, then without loss of generality $x_{s + 1}> x_1$ and we have $d \coloneqq d_1 - d_{s + 1} = x_{s + 1} - x_1 \in D$ and $d + d_2 + \dotsb + d_s = d_{s+2} + \dotsb + d_\ell$. By induction, there is a decomposition of the edges of the multigraph obtained from $M$ by deleting $e_1$ and $e_{s+1}$ and adding a new edge $e$ incident to $x_1$ and $x_{s+1}$. Adding back the edges $e_1$ and $e_{s+1}1$ (which form a path between $x_1$ and $x_{s+1}$) in place of $e$, we obtain a suitable decomposition of the edges of~$M$.    

For the `moreover' part of the claim, note that if $M$ has an odd number of edges, then the collection of circuits we obtain must contain a circuit with an odd number of edges. Any such circuit contains a cycle of odd length, which completes the proof of the claim.
\end{proof}

\subsection{Proof of Theorem~\ref{thm:oddgirth}}\label{sec:oddconstruction}

Fix a graph $F$ with at least one odd cycle, let $f = \abs{V(F)}$ and let $h$ be the odd girth of $F$. 
Let $p$ be a prime large enough that $\big[\lfloor p/h \rfloor\big]$ has a $B_h$-subset $S$ of size $f$ that contains $1$ and $2$. 
Let the elements of $S$ be $s_1 = 1, s_2 = 2, \dotsc, s_f$ in ascending order and let the difference set of $S$ be $D = \{s_j - s_i \colon i < j\}$. 
As $F$ has at least one edge we may take a copy $F^{\ast}$ of $F$ with vertex set $S$ which includes the edge $s_1s_2$. We write
\begin{equation*}
    E = \{s_j - s_i \colon i < j \text{ and } s_i s_j \in E(F^{\ast})\}
\end{equation*}
for the set of distances corresponding to edges of $F^{\ast}$. Since $s_1 = 1$ and $s_2 = 2$, we have $1 \in E$. Note that $E \subseteq D$. Let $\ov{D} = D + p\Z$ and $\ov{E} = E + p\Z$.

Let $\vv{X_n}$ be a digraph given by applying Lemma~\ref{lem:digraphexistence} with $g = h$. Recall that (a) the underlying graph $X_n$ has chromatic number $n$, (b) the orientation $\vv{X_n}$ is acyclic, (c) for any pair of vertices $u, v$, there is at most one directed path from $u$ to $v$ in $\vv{X_n}$, and (d) for every cycle in $X_n$, the corresponding oriented cycle in $\vv{X_n}$ has at least $h$ changes of direction.
For vertices $u < v$ (in the reachability ordering), write $d(u,v)$ for the length of the unique directed path from $u$ to $v$ in $\vv{X_n}$. We now define a digraph $\vv{G}$ with vertex set $V(\vv{X_n})$ as follows.
\begin{center}
    The edge $\vv{uv}$ is present if $u < v$ and $d(u, v) \in \ov{E}$.
\end{center}
As before, the digraph $\vv{G}$ produced is acyclic. We now show that $G$, its underlying undirected graph, satisfies the properties stated in Theorem~\ref{thm:oddgirth}. Firstly since $1 \in E$, the graph $X_n$ is a subgraph of $G$ so $\chi(G) \geq n$. Next we check that every induced subgraph of $G$ with sufficiently large chromatic number contains an induced copy of $F$. The argument is almost identical to the corresponding argument at the end of Section~\ref{sec:construction}. We colour each edge $uv$ of $G$ (and $\vv{G}$), where $u < v$, with the colour $d(u, v) \bmod{p}$. As before, any monochromatic directed path in $\vv{G}$ of length $pf$ gives rise to an induced copy of $F$. Let $H$ be an induced subgraph of $G$ which does not contain an induced copy of $F$. Each subgraph of $H$ consisting of $i$-coloured edges is $pf$-colourable by Proposition~\ref{prop:boundedpaths} and so, by a product colouring, $H$ is itself $c'_F$-colourable for $c'_F = (pf)^{\abs{E}}$.

It remains to check that the odd girth of $G$ is at least $h$. Suppose for a contradiction that $G$ contains an odd cycle $v_0v_1\dotsc v_{\ell - 1}$ of length $\ell < h$. Since $\vv{G}$ contains no directed cycles, without loss of generality the path $v_{\ell - 1}v_0v_1$ in $G$ is not a directed path in $\vv{G}$. We may further assume that since $\ell$ is odd, the path $v_0v_1v_2$ in $G$ is a directed path in $\vv{G}$ from $v_0$ to $v_2$. For each edge $v_i v_{i + 1}$ in this cycle (where here and throughout we take subscript addition in the $v_i$ to be modulo $\ell$), there is a directed path $\vv{P_i}$ from $v_i$ to $v_{i + 1}$ or from $v_{i + 1}$ to $v_i$ in $\vv{X_n}$ depending on the direction of the edge between $v_i$ and $v_{i + 1}$ in $\vv{G}$. For each $i$ let $P_i$ be the undirected path underlying $\vv{P_i}$ and concatenate these paths $P_i$ in the natural way to obtain a walk $W$ in $X_n$ which begins at $v_0$, visits $v_1, \dotsc, v_{\ell - 1}$ in turn, and then finally returns to $v_0$.

For each traversal of an edge in the walk, when the direction of traversal is the same as the direction of the edge in $\vv{X_n}$, we consider the traversal of this edge in the walk to be `coloured black'. When the two directions are different we consider the traversal of the edge to be `coloured red'. For each of the paths $P_i$, the walk's traversals corresponding to $P_i$ are all the same colour, and we have assumed that the colours for $P_0$ and $P_1$ are both black, so the walk changes colour at most $\ell - 2$ times.

Consider the subgraph $L$ of $G$ consisting of all vertices and edges contained in the walk~$W$. 
We examine two cases based on whether or not $L$ contains a cycle. 

\textbf{Case 1:} Suppose that $L$ is acyclic. Then $L$ is a tree. Let $e = uv$ be an edge of $L$. The graph $L-e$ (i.e.\ the graph with vertex set $V(L)$ and edge set $E(L)\setminus\{e\}$) has exactly two components, one of which contains $u$ and the other of which contains $v$. It is clear that the traversals of the edge $e$ in the walk $W$ alternate between traversing from $u$ to $v$ and traversing from $v$ to $u$. Moreover, since the walk starts and ends at $v_0$, there are an equal number of traversals of $e$ of in each direction. In other words, the total number of traversals of $e$ coloured black is equal to the total number coloured red. Since this is true for all edges $e$, the total number of black traversals of edges in the walk is equal to the total number of red traversals.

Let $S_1 = \bigl\{0 \leq i \leq \ell - 1 \colon \vv{v_iv_{i + 1}}\in E(\vv{G})\bigr\}$ and $S_2 = \bigl\{0 \leq i \leq \ell - 1 \colon \vv{v_{i + 1}v_i}\in E(\vv{G})\bigr\}$ so that $S_1$ and $S_2$ form a partition of $\{0,1,\dotsc,\ell - 1\}$. For $0 \leq i \leq \ell - 1$, let $d_i$ be the length of the path $P_i$. Then the total number of black traversals is equal to the sum of the $d_i$ for $i \in S_1$ and the total number of red traversals is the sum of the $d_i$ for $i\in S_2$. Hence by the above
\begin{equation*}
    \sum_{i\in S_1} d_i = \sum_{i\in S_2} d_i.
\end{equation*}
Since the edges $v_iv_{i + 1}$ are all present in $G$, we have $d_i \in \ov{E}$ for all $i$. Let $\ov{d_i} \equiv d_i \bmod{p}$ be chosen so that $0 \leq \ov{d_i} \leq p - 1$. Then $\ov{d_i} \in E$. The sums $\sum_{i \in S_1} \ov{d_i}$ and $\sum_{i \in S_2} \ov{d_i}$ are the same modulo $p$ and each have fewer than $h$ terms, all of which are at most $p/h$. It follows that, in fact,
\begin{equation*}
    \sum_{i \in S_1} \ov{d_i} = \sum_{i \in S_2} \ov{d_i}.
\end{equation*}
Noting that $E\subseteq D$, we can apply the cycle fact to show that the graph on vertex set $S$ with edges between pairs of vertices at distances $\ov{d_0},\dotsc,\ov{d_{\ell-1}}$ contains an odd cycle of length at most $\ell$. Since $\ov{d_0}, \dotsc, \ov{d_{\ell-1}} \in E$, by the definition of $E$ and the fact that $S$ is a Sidon set, $F^{\ast}$ contains a cycle of this length, which is a contradiction. This completes the analysis of the case where $L$ does not contain a cycle.

\textbf{Case 2:} Suppose that $L$ contains a cycle. Let $\Gamma$ be the first cycle produced by the walk, and label its vertices as $c_0, c_1, \dotsc, c_{r - 1}$ in cyclic order around $\Gamma$, where $c_0$ is the vertex amongst these that $W$ arrives at first. Let $C = \{c_0, \dotsc, c_{r - 1}\}$. We will show that $\Gamma$ contains fewer than $h$ changes of direction, which contradicts property (d) of $\vv{X_n}$. In what follows, we are only concerned with the portion of walk $W$ from its first visit to $c_0$ to the point when cycle $\Gamma$ is formed. Let $W'$ be this segment of $W$. In particular, $c_0$ is the first vertex of $W'$ and the final edge traversal in $W'$ completes cycle $\Gamma$.

Suppose that as we travel along $W'$, there is an occasion on which we arrive at $c_i$ for some $0\leq i\leq r-1$ and the next vertex in $C$ that we visit is $c_j$ for some $j\not \in\{i - 1,i,i + 1\}$, where here and throughout we take addition in the subscripts of the $c_i$ to be modulo~$r$. Note that $W'$ does not terminate when it reaches $c_j$ since the final edge it traverses must be an edge of $\Gamma$. This portion of the walk contains a path $P$ from $c_i$ to $c_j$ which avoids every vertex in $C\setminus\{c_i,c_j\}$. Consider the sets of edges $\{c_ic_{i + 1}, c_{i + 1}c_{i+2}, \dotsc, c_{j-1}c_{j}\}$ and $\{c_ic_{i - 1}, c_{i - 1}c_{i - 2}, \dotsc, c_{j+1}c_{j}\}$. Each of these sets of edges form a cycle with path $P$, so at the first time after the formation of $P$ at which all the edges in one set have been traversed, the graph walked so far contains a cycle. Since the union of these sets is the edge set of $\Gamma$, this occurs strictly before the formation of $\Gamma$, which is a contradiction. Thus, after $W'$ visits $c_i$, the next vertex of $C$ that it visits is one of $c_{i - 1}$, $c_i$ and $c_{i + 1}$.

Now suppose that after visiting $c_i$ for some $0\leq i\leq r - 1$, the next vertex of $C$ that walk $W'$ visits is $c_{i + 1}$, and suppose further that it arrives at $c_{i + 1}$ via an edge other than $c_ic_{i + 1}$. Clearly it also does not arrive at $c_{i + 1}$ via edge $c_{i+2}c_{i + 1}$, so it arrives via an edge not in $\Gamma$ and hence $W'$ does not terminate when it reaches $c_{i + 1}$. The portion of the walk between $c_i$ and $c_{i + 1}$ contains a path $P'$ from $c_i$ to $c_{i + 1}$ which avoids every vertex in $C\setminus\{c_i,c_{i + 1}\}$ and avoids the edge $c_ic_{i + 1}$. Similarly to above, at the first time after the formation of $P'$ at which either $c_ic_{i + 1}$ has been traversed or all of $c_ic_{i - 1},\dotsc,c_{i+2}c_{i + 1}$ have been traversed, the graph walked by $W'$ contains a cycle. This occurs strictly before the formation of $\Gamma$, which is a contradiction.

Similarly, if the next vertex of $C$ that $W'$ visits after $c_i$ is $c_{i - 1}$, then it arrives at $c_{i - 1}$ via the edge $c_ic_{i - 1}$. These facts combined imply that there are vertices $\gamma_0,\dotsc,\gamma_{s - 1}$ of $C$, with $\gamma_0=c_0$ and $\gamma_{j+1}$ a neighbour of $\gamma_j$ in $\Gamma$, such that $W'$ has the following form. It is a (possibly empty) walk from $\gamma_0$ to itself avoiding $C \setminus \{\gamma_0\}$, followed by a traversal of the edge $\gamma_0\gamma_1$, followed by a (possibly empty) walk from $\gamma_1$ to itself avoiding $C \setminus \{\gamma_1\}$, followed by a traversal of the edge $\gamma_1\gamma_2$, and so on, concluding with a traversal of the edge $\gamma_{s - 2}\gamma_{s - 1}$, which completes cycle $\Gamma$. Note that the vertices $\gamma_0, \gamma_1, \dotsc, \gamma_{s - 1}$ form a walk $W''$ on $\Gamma$ traversing all the edges of $\Gamma$. Let the traversals of each edge have the same colour in $W''$ as the corresponding traversals in $W'$. Note that $W''$ changes colour at most as many times as $W'$.

We say that a vertex $c_i\in C$ is \defn{crossed} by the walk $W''$ if $c_{i - 1}, c_i, c_{i + 1}$ are consecutive vertices of $W''$ in either order. If there are two vertices of $C$ not crossed by $W''$, then $W''$ does not traverse all edges of $\Gamma$, so there is at most one vertex not crossed by $W''$. Now, if the edges $c_{i - 1}c_i$ and $c_ic_{i + 1}$ are a change of direction, then when $W''$ crosses $c_i$ the traversals of these two edges are of different colours. It follows that every time $W''$ crosses $c_i$, it changes colour between the traversals of the two edges. Hence, since $W''$ changes colour at most $h-2$ times and all but at most one of the vertices in $C$ are crossed, there are at most $h-1$ changes of direction in $\Gamma$. We have thus obtained a cycle in $\vv{X_n}$ with fewer than $h$ changes of direction, which gives the required contradiction. Therefore, the odd girth of $G$ is at least $h$.

\section{Tournaments}\label{sec:tournaments}

We now prove Theorem~\ref{thm:trnmt}, which is restated below. Our strategy is akin to that used by Alon, Pach and Solymosi in~\cite{APS01}. Recall that a \defn{tournament} is a complete graph in which each edge has an orientation, and that a tournament is \defn{transitive} if, for all distinct vertices $u,v,w$, the presence of the edges $\vv{uv}$ and $\vv{vw}$ in the tournament implies the presence of the edge $\vv{uw}$. The \defn{chromatic number} $\chi(T)$ of a tournament $T$ is the smallest possible number of parts in a partition of its vertex set in which each part induces a transitive tournament. We will prove the following corollary of Theorem~\ref{thm:main}.

\trnmt*

Note that there are certain tournaments which are contained in every tournament of sufficiently large chromatic number; all such tournaments were described explicitly in~\cite{BCC...13}. Before proving the corollary, we state two results we will need. The first is the following standard corollary of Dilworth's theorem~\cite{dilworth}. 

\begin{lemma}\label{lem:dilworth}
Let $(a_1,a_2,\ldots, a_k)$ be a sequence of distinct elements and let $<$ be a strict total order on them. If this sequence has no decreasing \textnormal{(}with respect to $<$\textnormal{)} subsequence of length greater than $m$, then it can be partitioned into at most $m$ increasing subsequences.
\end{lemma}

We will also use the following result originally proved by R\"odl and Winkler in~\cite{RW89}. Recall that an \defn{ordering} of a graph or digraph $G$ is a strict total ordering of $V(G)$, and an \defn{ordered (di)graph} is a pair $(G,<)$ where $G$ is a (di)graph and $<$ is an ordering of $G$. 

\begin{theorem}[R\"odl and Winkler~\cite{RW89}]\label{thm:rodlwinkler}
For any ordered graph $(B,<)$, there exists a graph $F$ such that, for every ordering $<'$ of $F$, $(B,<)$ is an ordered induced subgraph of $(F,<')$. 
\end{theorem}

We are now ready to prove Theorem~\ref{thm:trnmt}.

\begin{proof}[Proof of Theorem~\ref{thm:trnmt}]
Fix an ordered tournament $(T, <)$ with vertices $x_1<\dots <x_{|T|}$. Let $(B,<)$ be the ordered back-edge graph of $(T,<)$, that is, the ordered graph on the same vertex set as $T$ and with the same ordering as $T$, where $x_ix_j\in E(B)$ for $i<j$ exactly when $\vv{x_jx_i}\in E(T)$. Let $F$ be the graph obtained by applying Theorem~\ref{thm:rodlwinkler} to $(B,<)$. 
Now apply Theorem~\ref{thm:main} to $F$. We obtain a constant $C_T=c_F$ and a graph $L$ of (arbitrarily large) chromatic number $n$ with $\omega(L)=\omega(F)=\omega$ such that every induced subgraph of $L$ with chromatic number greater than $C_T$ contains an induced copy of $F$. Fix an arbitrary ordering $<'$ of $L$ and define an ordered tournament $(S,<')$ with the same vertex set as $L$ and the same ordering as $L$ by orienting $\vv{xy}$ for $x<'y$ if $xy\not \in E(L)$ and $\vv{yx}$ otherwise (so that $(L,<')$ is the ordered back-edge graph of $(S,<')$).

We now show that $\chi(S)\geq n/\omega$, and hence that $\chi(S)$ is arbitrarily large.
Consider a transitive subtournament $A$ in $S$ and write its vertices as a sequence $(a_1,a_2, \dotsc, a_{\abs{A}})$ such that $\vv{a_ia_j}\in E(S)$ for each $i<j$.
Observe that this sequence has no decreasing (with respect to $<'$) subsequence of length $\omega+1$ as this would give rise to an $(\omega+1)$-clique in $L$. Hence, by Lemma~\ref{lem:dilworth}, we may partition the sequence into at most $\omega$ increasing subsequences. Each of these corresponds to an independent set in $L$ and hence, $\omega\cdot \chi(S)\geq n$ as required.

It remains to show that every subtournament $R$ of $S$ with $\chi(R)> C_T$ contains a copy of $T$. If $R$ is such a subtournament, then clearly the induced subgraph $R'$ of $L$ with the same vertex set as $R$ has $\chi(R')>C_T$. Hence, by our assumptions on $L$, $R'$ contains an induced copy of $F$. Consider the ordering of $F$ given by the restriction of $<'$ to this copy. By the construction of $F$, under this ordering it contains an induced ordered copy of $(B,<)$ which guarantees the existence of a copy of $T$ in $R$ as required. This completes the proof of the theorem.
\end{proof}

\section{Hypergraphs}\label{sec:hypergraphs}

In this section we prove Theorem~\ref{thm:hypergraphs}, which is restated below.
Recall that in this paper all hypergraph edges have size at least two. For a hypergraph $\cG$ and a set $X \subseteq V(\cG)$ of vertices of $\cG$, the \defn{subhypergraph of $\cG$ induced on $X$}, denoted $\cG[X]$, is the hypergraph with vertex set $X$ and edge set $\{e\colon e\subseteq X, e\in E(\cG)\}$. For hypergraphs $\cF$ and $\cG$, we will say that $\cG$ is \defn{$\cF$-free} if it does not contain $\cF$ as an induced subhypergraph. Recall from the introduction that a hypergraph is said to be \defn{strongly $t$-colourable} if its vertices can be $t$-coloured such that no edge contains two vertices of the same colour, and that a hypergraph \defn{covers} a pair of vertices if it has an edge containing both of them. 

\hypergraphs*

The proof of Theorem~\ref{thm:hypergraphs} resembles those of Theorems~\ref{thm:main} and~\ref{thm:oddgirth} in that we start with an oriented base (hyper)graph with large chromatic number and certain `reachability' properties, to which we add edges based on `distances' between vertices modulo some prime $p$. As before, we will ascertain which edges to add to the base hypergraph by placing a copy of $\cF$ on a $B_3$-set, an approach reminiscent of that used in the earlier proofs. Once we have constructed $\cG$, the arguments we use to show it has the desired properties also bear strong similarities to those used above.

We now introduce an analogue of reachability and distances between vertices in the hypergraph setting. Given a hypergraph $\cG$ and a total order $\prec$ on its vertex set, we construct the \defn{$\prec$-digraph} $\vv{G}$ on vertex set $V(\cG)$ by adding to $E(\vv{G})$ the directed edges $\vv{v_1 v_2}$, $\vv{v_2 v_3}$, \ldots, $\vv{v_{a - 1} v_a}$ for each edge $e = \{v_1 \prec v_2 \prec \dotsb \prec v_a\}$ of $\cG$. The \defn{$\prec$-graph} $G$ is the underlying undirected graph of $\vv{G}$. If $\cG$ has girth at least $3$, then no two distinct edges intersect in more than one vertex, and it follows that every pair of distinct vertices has at most one directed edge between them in $\vv{G}$. Also, in this setting every edge in $\vv{G}$ comes from a unique edge of~$\cG$.

The base hypergraph for our construction should have large chromatic number as well as an ordering $\prec$ on its vertex set such that its $\prec$-digraph is suitable for defining reachability and distance. The lemma below states that such hypergraphs exist.

\begin{lemma}\label{lem:nesrodhyp}
For all integers $k \geq 2$, $g \geq 3$ and $n \geq 2$, there exists a $k$-uniform hypergraph~$\cY_n$ with chromatic number $n$ and girth at least $g$, and an order $\prec$ on its vertex set such that, for every cycle in its $\prec$-graph $Y_n$, the corresponding oriented cycle in its $\prec$-digraph $\vv{Y_n}$ has at least $g$ changes of direction.
\end{lemma}

\begin{proof}
    Ne\v{s}et\v{r}il and R\"odl~\cite{NR79} proved that for all integers $k \geq 2$, $g \geq 3$ and $n \geq 2$ there is a $k$-uniform hypergraph with chromatic number $n$ and girth at least $g$ which is strongly $a$-colourable\footnote{Although the strongly $a$-colourable condition does not appear in their theorem statement, it is emphasised in their proof.} for $a = (k - 1) n + 1$. 
    
    Fix integers $k \geq 2$, $g \geq 3$, and $n \geq 2$, and let $a = (k - 1)n + 1$. By the result of Ne\v{s}et\v{r}il and R\"odl, there is a $k$-uniform hypergraph $\cY_n$ with chromatic number $n$ and girth greater than $(g - 1)(a - 1)$ which is strongly $a$-colourable. Fix a strong $a$-colouring $c \colon V(\cY_n) \to \{1, \dotsc, a\}$ of $\cY_n$. Let $\prec$ be any ordering of $V(\cY_n)$ in which every vertex of colour $i$ precedes every vertex of colour $j$ whenever $i < j$. Let $Y_n$ and $\vv{Y_n}$ respectively be the $\prec$-graph and $\prec$-digraph of $\cY_n$. 
    
    Consider any edge $\vv{uv}$ in $\vv{Y_n}$. There is some edge of $\cY_n$ containing both $u$ and $v$ and so $c(u) \neq c(v)$. By the choice of $\prec$, we must have $c(u) < c(v)$. Letting $u_1 u_2 \dotsc u_t$ be any directed path in $\vv{Y_n}$, by the preceding discussion we have $1 \leq c(u_1) < \dotsb < c(u_t) \leq a$. Hence, the length of any directed path in $\vv{Y_n}$ is at most $a - 1$. 
    
    Consider any cycle $C$ in $Y_n$: each edge of this cycle comes from an edge of a cycle in $\cY_n$ and so $C$ has length greater than $(g - 1)(a - 1)$. Every directed path in $\vv{Y_n}$ has length at most $a - 1$, so $C$ has at least $g$ changes of direction in $\vv{Y_n}$, as required.
\end{proof}

We are now ready to prove Theorem~\ref{thm:hypergraphs}.

\begin{proof}[Proof of Theorem~\ref{thm:hypergraphs}]
Fix a hypergraph $\cF$ with at least one edge. Let $f$ be the number of vertices of $\cF$ and let $n \geq 2$. Let $m \geq 2$ be the minimum edge size of $\cF$. Let $\cY_n$ be an $m$-uniform hypergraph of chromatic number at least $n$ and girth at least four given by Lemma~\ref{lem:nesrodhyp}, with an ordering $\prec$ on its vertex set such that in the associated digraph $\vv{Y_n}$ every cycle has at least three changes of direction. In particular, in $\vv{Y_n}$ there are no directed cycles and there is at most one directed path between any pair of vertices. If there is a directed path from $u$ to $v$ in $\vv{Y_n}$, then we write $u < v$ and define the distance between $u$ and $v$, written $d(u,v)$, to be the length of this path.

We will define a hypergraph $\cG$ satisfying the conditions of Theorem~\ref{thm:hypergraphs} by adding edges to $\cY_n$ as follows. Start by picking a prime $p$ such that $[\lfloor p/(4m)\rfloor]$ contains a $B_3$-set $S'$ of size $f$, then define $S = 2m \cdot S'=\{2ms'\colon s'\in S'\}\subseteq [(p-1)/2]$. Let the elements of $S$ be $s_1, \dotsc, s_f$ in ascending order and let $D = \{s_j - s_i \colon i < j\}$ be the difference set of $S$. Note that all elements of $D$ are positive multiples of $2m$ and are at most $(p - 1)/2$.

Taking a copy $\cF^{\ast}$ of $\cF$ with vertex set $S$, we define
\begin{equation*}
    E = \{(s_{i_2} - s_{i_1}, s_{i_3} - s_{i_2}, \dotsc, s_{i_a} - s_{i_{a - 1}}) \colon i_1 < i_2 < \dotsb < i_a \textnormal{ and } \{s_{i_1},s_{i_2}, \dots, s_{i_a}\} \in E(\cF^{\ast})\}
\end{equation*}
to be the analogue of the sets of allowable distances from the proofs of Theorems~\ref{thm:main} and~\ref{thm:oddgirth}. Note that every entry of every tuple in $E$ is in $D$. Let $\ov{D} = D + p\Z$ and $\ov{E} = \{(d_1 + k_1 p, \dotsc, d_{a - 1} + k_{a - 1} p)\colon k_1, \dotsc, k_{a - 1} \in \Z \textnormal{ and } (d_1, \dotsc, d_{a - 1})\in E\}$. Construct hypergraph $\cG$ by \emph{adding} edges to $\cY_n$ as follows.
\begin{center}
    The edge $\{v_1, v_2, \dotsc, v_a\}$ is added if $v_1 < v_2 < \dotsb < v_a$ and $(d(v_1, v_2), d(v_2, v_3), \dotsc, d(v_{a - 1}, v_a)) \in \ov{E}$.
\end{center}

There are two types of edges in $\cG$: those which appear in $\cY_n$ and the newly added ones. By construction, if two vertices are in a common edge, then they are $<$-comparable. The edges which appear in $\cY_n$ are of the form $\{v_1,\dots,v_m\}$ where $d(v_i,v_{i + 1})=1$ for all $i$. Hence, if vertices $u < v$ are both in some edge of $\cY_n$, then $d(u, v) \in [m - 1]$. On the other hand, if $u < v$ are both in an edge that has been newly added, then $d(u, v) \in \ov{D}$ and so $d(u, v) \geq 2m$. In particular, if vertices $u < v$ are both in some edge of $\cG$, then $d(u, v)$ is in $\ov{D} \cup [m - 1]$ and so is not divisible by $p$.

We will now show that this hypergraph satisfies the conditions of Theorem~\ref{thm:hypergraphs}. Firstly, since it contains $\cY_n$, it certainly has chromatic number at least $n$. Next, we will show that every $\cF$-free induced subhypergraph of $\cG$ has bounded strong chromatic number. We begin by defining $L$ to be the graph with vertex set $V(\cG)$ where $uv$ is an edge if $u$ and $v$ are in some common edge of $\cG$. For each edge $uv$ of $L$, $u$ and $v$ are $<$-comparable. Orient the edge from $u$ to $v$ if $u < v$, and from $v$ to $u$ otherwise. Denote the resulting digraph by $\vv{L}$.

All edges $\vv{uv}$ of $\vv{L}$ have endpoints which satisfy $d(u,v)\in \ov{D}\cup [m-1]$, and we colour the edges of $\vv{L}$ (and $L$) by the residue modulo $p$ of this distance. We claim that the vertex set of a monochromatic directed path of length $pf$ in $\vv{L}$ has a subset which induces a copy of $\cF$ in $\cG$. Indeed, by repeatedly jumping at most $p-1$ steps along such a path we can find a sequence of vertices $u_1 < u_2 < \dotsb < u_f$ in the path such that $d(u_i, u_{i + 1}) \equiv s_{i + 1} - s_i \bmod{p}$ for all $i$. Note that, for $i < j$, $d(u_i, u_j) \equiv s_j - s_i \bmod{p}$ and so $d(u_i, u_j) \in \ov{D}$. Hence, $d(u_i, u_j)$ is never 1 and so $\{u_1, \dotsc, u_f\}$ does not contain an edge of $\cY_n$. Thus, by the construction of $\cG$, the vertices $u_1, \dotsc, u_f$ induce a copy of $\cF$, as claimed.

Now let $\cH$ be an $\cF$-free induced subhypergraph of $\cG$. By the above, for each $i\in D\cup [m-1]$ there is no $i$-coloured directed path of length $pf$ in $\vv{L}[V(\cH)]$, so the $i$-coloured subgraph of $L[V(\cH)]$ is $(pf)$-colourable by Proposition~\ref{prop:boundedpaths}. A product colouring demonstrates that $L[V(\cH)]$ is $(pf)^{p - 1}$-colourable.
This colouring is a strong colouring of $\cH$. Hence we may take $c_{\cF} = (pf)^{p - 1}$, which is a constant that only depends on $\cF$.

It remains to show the moreover part of the statement. Let $X = \{v_1,\dots,v_t\}$ be a set of vertices every pair of which is covered by $\cG$. If $t\leq 1$, then the result is clear so assume $t\geq 2$. It follows from the fact that $\cG$ covers every pair from $X$ that each pair is $<$-comparable; we may assume that $v_1 < \dots < v_t$. Define $d_i=d(v_i,v_{i + 1})$ for all $i$ and note that since $d(v_i, v_j) \in \ov{D} \cup [m - 1]$ for all $i < j$, all sums of the form
\begin{equation}\label{eq:hyp}
    \sum_{i_1\leq j\leq i_2} d_j
\end{equation}
with $1 \leq i_1 \leq i_2 \leq t-1$ are in $\ov{D}\cup [m-1]$.

Suppose that $d_i\leq m-1$ for some $i$. If $i<t-1$, then $d_{i + 1}$ and $d_{i}+d_{i + 1}$ differ by at most $m-1$. This cannot occur if $d_{i + 1}\in \ov{D}$, so we conclude that $d_{i + 1}\leq m-1$. Similarly, if $i>1$ then $d_{i - 1}\leq m-1$. Repeating this argument we find that $d_1,\dots,d_{t-1}\in [m-1]$. Since the minimal positive element of $\ov{D}$ is at least $2m$, it follows that $d_1+d_2\not\in \ov{D}$, so $d_1+d_2\in [m-1]$. Applying this repeatedly we find that all sums of the form in~\eqref{eq:hyp} are in $[m-1]$, and hence no pair of vertices from $X$ appear together in an edge of $\cG$ which is not in $\cY_n$. In particular, $\cY_n$ covers all pairs from $X$. 

Let $e$ be an edge of $\cY_n$ containing both $v_1$ and $v_2$, and suppose there is some $v_a$ not in $e$. Let $e_1$, $e_2$ be edges of $\cY_n$ such that $e_i$ contains $v_a$ and $v_i$. If $e_1 = e_2$, then the edges $e$ and $e_1$ form a 2-cycle in $\cY_n$, while if $e_1 \neq e_2$, then the edges $e, e_1, e_2$ form a 3-cycle in $\cY_n$. As $\cY_n$ has girth at least four, neither of these cases is possible. We deduce that $X \subseteq e$, and in particular that $t \leq m$. If $t < m$, then the subhypergraph of $\cG$ induced on $X$ has no edges and thus is clearly an induced subhypergraph of $\cF$. If $t = m$, then the subhypergraph of $\cG$ induced on these vertices has a single edge of size $m$, which again is clearly an induced subhypergraph of $\cF$.

On the other hand, if $d_i > m - 1$ for all $i$, then all the sums of the form in~\eqref{eq:hyp} are in $\ov{D}$. Writing $\ov{d_i}$ for the residue modulo $p$ of $d_i$, we have $\ov{d_i}\in D$ for all $i$ and all sums of the form \begin{equation*}
    \sum_{i_1\leq j\leq i_2} \ov{d_j}
\end{equation*}
with $1 \leq i_1 \leq i_2 \leq t-1$ are in $\ov{D}$. We can deduce from the fact that $D\subseteq [(p-1)/2]$ that in fact all of these sums are in $D$. It now follows from the clique fact that there exist $b_1<\dots<b_t$ in $S$ such that $b_{i + 1} - b_i =\ov{d_i}$ for all $i$. Since $\ov{d_i}\neq 1$ for all $i$, the subhypergraph of $\cG$ induced on $X$ does not include any edge of $\cY_n$, so by the construction of $\cG$ and the fact that $S$ is a Sidon set, the subhypergraph of $\cG$ induced on $X$ is isomorphic to that of~$\cF^{\ast}$ induced on $\{b_1,\dots,b_t\}$, and the claim follows.
\end{proof}

\section{Extensions to infinite families of graphs}\label{sec:infinite}

Recall that the \defn{disjoint union} of a family of graphs is the (possibly infinite) graph consisting of pairwise vertex-disjoint copies of the graphs in the family with no edges between copies. As noted in the introduction, for a finite family of graphs $\cF$ (at least one of which contains an edge), an easy consequence of Theorem~\ref{thm:main} is that there is a constant $c_{\cF}$ and graphs $G$ of arbitrarily large chromatic number and the same clique number as the largest clique number of a graph in $\cF$ such that every induced subgraph of $G$ with chromatic number greater than $c_{\cF}$ contains every member of $\cF$ as an induced subgraph. Taking a disjoint union of the graphs $G$ of arbitrarily large chromatic number gives the following corollary.

\begin{corollary}\label{cor:universalinfinite}
    Let $\omega \geq 2$, and suppose that $\cF$ is a finite family of finite graphs all with clique number at most $\omega$. There exists a constant $c_{\cF}$ and an infinite graph $G$ of infinite chromatic number and clique number at most $\omega$ such that every induced subgraph of $G$ with chromatic number greater than $c_{\cF}$ contains every member of $\cF$ as an induced subgraph.
\end{corollary}

Does this phenomenon occur for infinite families of graphs? To this end we will say that a (possibly infinite) family of graphs $\cF$ is \defn{durable}
if there exists an infinite graph $G$ with infinite chromatic number such that every infinite-chromatic induced subgraph of $G$ contains every member of $\cF$ as an induced subgraph. If such a $G$ exists we will call it a \defn{witness} of $\cF$'s durability. We note a few properties of witnesses and durability.
\begin{itemize}[noitemsep]
    \item A subfamily of a durable family is durable.
    \item Every infinite-chromatic induced subgraph of a witness is also a witness.
    \item Every witness contains every graph in $\cF$ as an induced subgraph.
\end{itemize}
Corollary~\ref{cor:universalinfinite} says that every finite family of graphs is durable. Moreover, if at least one graph in the finite family contains an edge, then there is a witness whose clique number is no bigger than the largest clique number of a member of the family.

We are interested in which infinite families of finite graphs are durable. There certainly are some durable infinite families, for example any family whose members are all disjoint unions of complete graphs is durable as witnessed by the disjoint union of $K_1$, $K_2$, $K_3$, \ldots. This example suggests that for a durable family with unbounded chromatic number, the disjoint union of the members of the family might be a witness of the family's durability.
If this were the case, then this disjoint union would be a minimal witness. The following result confirms this hypothesis.

\begin{theorem}\label{thm:durable}
    Let $\cF$ be a countable family of finite graphs. Let $G_{\cF}$ be the disjoint union of the members of $\cF$. 
    \begin{enumerate}[label = \textnormal{(\alph{*})}, noitemsep]
        \item If $\cF$ is durable, then every witness of $\cF$'s durability contains~$G_{\cF}$ or $K_{\infty}$ as an induced subgraph.
        \item If $\{\chi(F) \colon F \in \cF\}$ is unbounded, then $\cF$ is durable if and only if $G_{\cF}$ is a witness of $\cF$'s durability.
    \end{enumerate}
\end{theorem}

Theorem~\ref{thm:durable}(b) shows that if a family of unbounded chromatic number is durable, then it has a witness whose clique number is no bigger than the largest clique number of a member of the family. An immediate application of Theorem~\ref{thm:durable} is Theorem~\ref{thm:girthbad}: the family, $\cF_g$, of finite graphs of girth at least $g$ is not durable. Indeed, this family has unbounded chromatic number and so if it were durable, then $G_{\cF_g}$ would witness this. However, the disjoint union of graphs of girth at least $g + 1$ is an infinite-chromatic induced subgraph of $G_{\cF_g}$ that does not contain the cycle of length $g$.

The following lemma will be needed in the proof of Theorem~\ref{thm:durable}.

\begin{lemma}\label{lemma:unbddclique}
    Every infinite graph either contains an infinite clique, has bounded clique number or contains the disjoint union of all finite cliques as an induced subgraph.
\end{lemma}

\begin{proof}
    Let $G$ be an infinite graph with unbounded clique number. Let $(a_i)_{i\geq 1}$ be a sequence of natural numbers in which each term is taken to be large relative to all previous terms. Fix a copy of $K_{a_1}$ in $G$ on vertex set $V_1$. By identifying a copy of $K_{a_1+a_2}$ in $G$, we can find a copy of $K_{a_2}$ on vertex set $V_2$ disjoint from $V_1$. Continuing in this manner, we obtain a sequence $(V_i)_{i\geq 1}$ of pairwise disjoint sets of vertices of $G$ such that $V_i$ induces a copy of $K_{a_i}$ for each $i$.

    Select a vertex $v\in V_1$, and for each $i\geq 2$ remove at most half of the vertices from $V_i$ so that $v$ is either adjacent to every remaining vertex or not adjacent to any of them. Repeat this process for each of the other vertices in $V_1$, removing at most half of the vertices currently in each $V_i$ at each step, so that at the end of the process the remaining sets of vertices $V'_i\subseteq V_i$ satisfy $\abs{V'_i}\geq \abs{V_i}/2^{a_1}$ and have the property that for all $i\geq 2$ and $v\in V_1$, $v$ is either adjacent to every vertex in $V'_i$ or not adjacent to any of them.

    For each $i\geq 2$ we can now define a vector $b_i\in \{0,1\}^{a_1}$ by letting the $j$th entry be $1$ if the $j$th vertex of $V_1$ is adjacent to every vertex in $V'_i$ and $0$ otherwise. There are only finitely many such vectors, so there exists a sequence $(i_k)$ such that all vectors $b_{i_k}$ agree. By relabelling, we may therefore assume that all vectors $b_i$ are equal to some $b$. If at least half the entries in $b$ are $1$'s, then by considering the vertices in $V_1$ corresponding to $1$ entries, for some $t\geq a_1/2$ we obtain a copy of $K_{t}$ in $G$, every vertex of which is adjacent to every vertex in $\bigcup_{i\geq 2}V'_i$. Otherwise, for some $t\geq a_1/2$ we obtain a copy of $K_t$ none of whose vertices have a neighbour in this union.

    Repeat this process for $V'_2$, then what remains of $V'_3$, and so on. If the $a_i$ are chosen appropriately, then it follows that there exists a sequence of pairwise disjoint sets of vertices of $G$, $(W_n)_{n\geq 1}$, such that $W_n$ induces a copy of $K_n$ in $G$ and the possible edges between $W_n$ and $\cup_{i> n}W_i$ are either all present in $G$ or all not present in $G$. If these edges are present for infinitely many $n$, then $G$ contains an infinite clique. If the edges are present for only finitely many $n$, then there is an induced subgraph of $G$ consisting of a disjoint union of arbitrarily large cliques, which in particular contains the disjoint union of all finite cliques as an induced subgraph.
\end{proof}

We are ready to prove Theorem~\ref{thm:durable}. For a vertex $v$ in a (possibly infinite) graph $G$, let $\Gamma(v) \coloneqq \{u\in V(G)\colon uv\in E(G)\}$ be the \defn{neighbourhood} of $v$ in $G$. Similarly, for $V \subseteq V(G)$ we write $\Gamma(V) \coloneqq \{u\in V(G)\colon uv\in E(G) \;\text{for some}\; v\in V\}$.

\begin{proof}[Proof of Theorem~\ref{thm:durable}]
    We first show that (a) implies (b) in the statement of the theorem. Suppose that $\cF$ is durable with $\{\chi(F)\colon F\in \cF\}$ unbounded, and  let $G$ be a witness of $\cF$'s durability. By part (a), $G$ contains $G_{\cF}$ or $K_{\infty}$ as an induced subgraph. In the former case, since $\{\chi(F) \colon F \in \cF\}$ is unbounded, $G_{\cF}$ has infinite chromatic number and is therefore a witness by one of the properties of witnesses noted above. In the latter case, again by the properties noted above, $K_{\infty}$ is a witness of $\cF$'s durability and hence every member of $\cF$ is an induced subgraph of~$K_{\infty}$. Thus, $\cF$ is a family of cliques and it follows from the fact that $\{\chi(F) \colon F \in \cF\}$ is unbounded that $G_{\cF}$ is a witness. 
    
    It remains to prove the first part of the theorem. Suppose that $\cF$ is durable and let $G$ be a witness. By Lemma~\ref{lemma:unbddclique}, $G$ either contains an infinite clique, has bounded clique number or contains the disjoint union of all finite cliques as an induced subgraph. In the first case we are done. In the third case, the disjoint union of all finite cliques is a witness of $\cF$'s durability, so all members of $\cF$ are disjoint unions of cliques. In particular, the disjoint union of all finite cliques contains $G_{\cF}$ as an induced subgraph and so we are done.
    
    Hence, we are left with the case where $G$ has bounded clique number. We may assume that the graph induced on the neighbourhood of each vertex of $G$ is finitely colourable. Indeed, otherwise there exists $v \in V(G)$ such that $G[\Gamma(v)]$ is an infinite-chromatic induced subgraph of $G$. In this case $G[\Gamma(v)]$ is a witness of $\cF$'s durability which has clique number strictly less than that of $G$. Repeating this process finitely many times, we obtain an infinite-chromatic induced subgraph of $G$ that is a witness of $\cF$'s durability and in which the neighbourhood of every vertex is finitely colourable.
    
    Enumerate $\cF$ as $(F_i)_{i \geq 1}$. Since $G$ is a witness, $G$ contains an induced copy of $F_1$, say on vertex set $V_1$. By assumption, the chromatic number of the subgraph of $G$ induced on $V_1\cup\Gamma(V_1)$ is finite. It follows that $G - (V_1\cup\Gamma(V_1))$ is an infinite-chromatic induced subgraph of $G$ and therefore contains every graph in $\cF$ as an induced subgraph. We may now find an induced copy of $F_2$ in this graph, say on vertex set $V_2$, and remove every vertex in $V_2$ or its neighbourhood from $G-(V_1\cup\Gamma(V_1))$ to obtain a new infinite-chromatic induced subgraph of $G$. Continuing in this manner we obtain an induced subgraph of $G$ consisting of the disjoint union of the members of $\cF$, that is, $G$ contains $G_{\cF}$ as an induced subgraph.
\end{proof}

Theorem~\ref{thm:durable} gives an effective test for whether a family of unbounded chromatic number is durable. Still, it is not clear whether there are durable families of unbounded chromatic number whose members are not all disjoint unions of cliques. Theorem~\ref{thm:durable} shows that the existence of such a family is equivalent to a positive answer to the following.

\begin{question}\label{qu:sequencedurable}
    Does there exist a sequence of connected finite graphs $(F_i)_{i\geq 1}$, not all of which are cliques, and positive integers $(c_i)_{i\geq 1}$ such that $\chi(F_i) \to \infty$ and every $F_i$-free subgraph of $F_j$ is $c_i$-colourable for all $i<j$?
\end{question}

We now turn to infinite families with bounded chromatic number. The next result provides some durable families of forests. A graph $F$ is \defn{$\chi$-bounding} if the class of $F$-free graphs is $\chi$-bounded. It is folklore that any $\chi$-bounding graph must be a forest and the celebrated Gy\'{a}rf\'{a}s--Sumner conjecture~\cite{Gyarfas75,Sumner81} asserts the converse. We remark that a beautiful proof of Gy\'{a}rf\'{a}s~\cite{Gyarfas87} shows that every path is $\chi$-bounding and so the following theorem, combined with the first bullet point above, implies Theorem~\ref{thm:pathdurable}.

\begin{theorem}\label{thm:chidurable}
    The family of $\chi$-bounding graphs is durable.
\end{theorem}

\begin{proof}
    Let $G$ be a triangle-free graph with infinite chromatic number. Let $H$ be an induced subgraph of $G$ with infinite chromatic number. Fix a $\chi$-bounding graph $F$ and note that there is some constant $c_F$ such that every triangle-free $F$-free graph is $c_F$-colourable. But $H$ is triangle-free and has chromatic number greater than $c_F$, so it contains $F$ as an induced subgraph. Hence, $G$ witnesses the durability of the family of $\chi$-bounding graphs.
\end{proof}

It is possible that every family with bounded chromatic number is durable (although we suspect this is not the case). An answer to the following very natural question would resolve this.

\begin{question}\label{qu:kcoldurable}
    Is the family of $k$-colourable graphs durable? That is, for each positive integer $k$, is there a graph $G_k$ with infinite chromatic number such that every infinite-chromatic induced subgraph of $G$ contains every finite $k$-colourable graph as an induced subgraph?
\end{question}

It would be natural to ask in addition for the graph $G_k$ to have clique number $k$ to align with Corollary~\ref{cor:universalinfinite}.

\section{Conclusion}\label{sec:conclusion}

Having considered tournaments, hypergraphs and infinite families of graphs, we now focus on the setting of graphs. As discussed in the introduction, Theorem~\ref{thm:main} shows that the class $\cC_r$ of $K_r$-free graphs has a far stronger property than just being a vertex Ramsey class, and it is natural to ask which other families have this property.

\begin{question}\label{qu:Ramsey}
    Which hereditary graph classes $\mathcal{C}$ have the property that for every $F \in \mathcal{C}$ there is a constant $c_F$ and graphs $G \in \mathcal{C}$ of arbitrarily large chromatic number such that every $F$-free induced subgraph of $G$ is $c_F$-colourable?
\end{question}

A good starting point for tackling Question~\ref{qu:Ramsey} would be to resolve the case when $\mathcal C$ is determined by a single excluded graph. 

A case of particular interest is when~$\cC$ is the class of graphs with girth at least $g$ for some $g\geq 3$. Theorem~\ref{thm:oddgirth} asserts that the class of graphs with \emph{odd girth} at least $g$ has the stated property. As mentioned in the introduction we conjecture that the same is true for girth.

\begin{conjecture}\label{conj:girth}
    For every graph $F$ with at least one cycle, there exists a constant $b_F$ and graphs $G$ of arbitrarily large chromatic number and the same girth as $F$ such that every $F$-free induced subgraph of $G$ is $b_F$-colourable.
\end{conjecture}

For this conjecture to hold, the class of graphs with girth at least $g$ must certainly be vertex Ramsey. Happily, since cycles are 2-connected this follows from a theorem of Ne\v{s}et\v{r}il and R\"odl~\cite{NR76}.

The graphs $\vv{X_n}$ and $X_n$ described in Section~\ref{sec:base} can be chosen to not contain short cycles. However, if we continue to build our graphs by adding edges $\vv{uv}$ for $u < v$ simply based on residues modulo $p$ of $d(u,v)$, we will unavoidably introduce many 4-cycles. For this reason, a resolution of this conjecture even in the special case where $F$ is the 5-cycle would be very interesting, as such an argument would likely overcome many of the difficulties involved in proving the full conjecture.

Another area for further research is determining the optimal constant $c_F$ in the statement of Theorem~\ref{thm:main}.

\begin{question}\label{q:bestbound}
    For each graph $F$ containing an edge, what is the smallest value of $c_F$ for which Theorem~\ref{thm:main} holds? Can we take $c_F$ bounded by a function of $\chi(F)$?
\end{question}

Clearly $c_F$ must be at least $\chi(F)-1$ for all $F$, and at least $\chi(F)$ if $F$ is not vertex-critical.  
Ne\v{s}et\v{r}il~\cite{Nesconj} has conjectured that the optimal $c_F$ is at most $\chi(F)$ for all $F$.

For triangles, it follows from results in~\cite{CHMS22,SS16} that the optimal $c_{K_3}$ is either 3 or 4 (in particular, the bound $\chi(F) - 1$ for vertex-critical graphs is not always attained). For larger cliques, Bria\'nski, Davies and Walczak~\cite{BDW23} showed that the optimal $c_{K_m}$ is $O(m^3)$; as noted at the start of Section~\ref{sec:mainproof}, the optimal $c_F$ is $O(\abs{V(F)}^9)$ in general.

\section*{Acknowledgements}
The authors would like to thank Jarik Ne\v{s}et\v{r}il for helpful discussions and comments, including pointing out a simpler way of obtaining the base structures in Section~\ref{sec:base} and Lemma~\ref{lem:nesrodhyp}, and for suggesting Question~\ref{q:bestbound}.

{\fontsize{11pt}{12pt}
\selectfont
}

\end{document}